\documentclass[10pt]{amsart}

\usepackage{amsfonts}
\usepackage{amsmath}
\usepackage{amssymb}
\usepackage{amsthm}
\usepackage{enumerate}
\usepackage{mdwlist}
\usepackage{mathrsfs}
\usepackage{mathabx}
\usepackage{url}
\usepackage{graphicx}
\usepackage{latexsym}
\usepackage{stmaryrd}
\usepackage{verbatim}
\usepackage{xcolor}

\usepackage{mathtools}
\usepackage{soul}
\usepackage{float}
\usepackage{tikz}
\usepackage{bbm}

\newtheorem{theorem}{Theorem}[section]
\newtheorem{proposition}[theorem]{Proposition}
\newtheorem{lemma}[theorem]{Lemma}
\newtheorem{corollary}[theorem]{Corollary}

\theoremstyle{remark}

\newtheorem{remark}[theorem]{Remark}
\newtheorem{example}[theorem]{Example}

\newcommand{\NN}{\mathbb{N}}
\newcommand{\ZZ}{\mathbb{Z}}
\newcommand{\RR}{\mathbb{R}}
\newcommand{\PP}{\mathbb{P}}
\newcommand{\EE}{\mathbb{E}}
\newcommand{\var}{\mathrm{Var}}

\newcommand{\norm}[1]{\left\lVert#1\right\rVert}

\colorlet{DarkGreen}{green!40!black!}

\begin{document}

\title[An approximate zero-one law]{An approximate zero-one law via the Dialectica interpretation}

\author[Thomas Powell and Alex Wan]{Thomas Powell and Alex Wan}
\date{\today}

\begin{abstract}
Zero-one laws state that probabilistic events of a certain type must occur with probability either $0$ or $1$, and nothing in between. We formulate a syntactic zero-one law, which enjoys good logical properties while being broadly applicable in probability theory. Then, inspired by G\"odel's Dialectica interpretation, we finitise it: The result is an approximate zero-one law which states that events with a particular finite structure occur with probability close to $0$ or $1$ up to an arbitrary degree of precision. This approximate zero-one law is equivalent - over classical logic - to the original zero-one law, but in contrast to the latter, is formulated entirely in terms of finite unions and intersections of events. Furthermore, in line with recent logical metatheorems for probability, it admits a computational interpretation, which in turn facilitates a quantitative analysis of theorems whose proof makes use of zero-one laws. Concrete applications in this spirit, over a variety of different settings, are discussed.
\end{abstract}

\maketitle

\section{Introduction}
\label{sec:intro}

Zero-one laws assert that events $A$ of a particular kind occur with probability $0$ or $1$. A canonical example is the Kolmogorov zero-one law, which states that if $(A_n)$ is a sequence of independent events, whenever an event $A$ belongs to the associated tail algebra, that is
\[
A\in \bigcap_{n=0}^\infty \mathcal{T}_n \ \ \mbox{for} \ \ \mathcal{T}_n:=\sigma(A_{n},A_{n+1},\ldots),
\] 
then we can conclude that $\PP(A)\in \{0,1\}$. The Kolmogorov zero-one law, along with its numerous variants and extensions (such as L\'evy's zero-one law for martingales and the Hewitt-Savage zero-one law for symmetric events), has applications throughout probability theory, including the convergence of random series and percolation theory. Indeed, zero-one laws are so fundamental that one tends to find them wherever stochasticity features over some infinite structure. For example, a number of zero-one laws exist in the context of logic, notably those relating to first-order sentences in finite models \cite{Fagin:76:probabilitiesonfinitemodels,GlebskiiKoganLiogon'kii:72:realizabilitypredicatecalculus}, a line of research that continues to be developed (e.g.\ \cite{GraedelHelalNaafWilke:22:zeroonesemiring}).

We take a first step towards tackling the question of whether it is possible to interpret zero-one laws in a `finitary', or computational manner. This is a challenging problem: Zero-one laws tell us that one of two things happen, but do not necessarily provide any insight into \emph{which one}. Indeed, in general it is not possible to effectively decide which of the two cases hold, and this is often a deep mathematical question in its own right, requiring insights far beyond those of the initial zero-one law (a classic example being Kesten's celebrated result \cite{kesten:80:half} on the critical threshold for bond percolation in $\ZZ^2$). Even when this information is available, the individual statements $\PP(A)=0$ and $\PP(A)=1$ are typically themselves infinitary in nature, and, as we will show, in the simplest of settings cannot in general be given a direct computational interpretation.

We approach this problem via the logic-based proof mining program \cite{kohlenbach:08:book}. First, we formulate a new zero-one law specifically designed for situations where the underlying event can be put into a particularly nice syntactic form. This result arises as a natural generalisation of the zero-one laws induced by the Borel-Cantelli lemmas and Erd\H{o}s-R\'enyi theorem, the latter having been recently studied from a logical perspective in \cite{arthan-oliva:21:borel-cantelli}, but covers a variety of situations in which the Kolmogorov zero-one law would traditionally be invoked. We then use G\"odel's Dialectica interpretation \cite{goedel:58:dialectica} to produce a finitization (in the sense of Tao \cite{tao:07:softanalysis}) of our zero-one law. The resulting finitization can be given a direct, quantitative interpretation, analogous to the so-called rates of metastability that are widely seen in proof mining (where they are typically associated with convergence statements). More interestingly, it constitutes a new example of the general phenomenon, recently made precise by Neri and Pischke \cite{neri-pischke:pp:formal}, of the extractability of uniform computational information from proofs in probability that can be formulated in terms of probability contents (i.e.\ finite unions and intersections). We illustrate our finitary zero-one law with three short case studies covering, respectively, the Erd\H{o}s-R\'enyi theorem, convergence for series of random variables, and bond percolation, the latter appearing for the first time in the context of proof mining.

Our work forms a contribution to the general program of bringing proof mining to bear on probability theory. While up until very recently this was limited to isolated case studies \cite{arthan-oliva:21:borel-cantelli,avigad-dean-rute:12:dominated,avigad-gerhardy-towsner:10:local}, the last years have seen a sustained, systematic approach to applying proof theory in probability, including the aforementioned logical metatheorems \cite{neri-pischke:pp:formal}, along with a series of applications focused on the theory of martingales \cite{neri-powell:pp:rs,neri-powell:25:martingale}, modes of stochastic convergence \cite{neri-pischke-powell:25:learnability}, laws of large numbers \cite{Ner2024,neri:25:kronecker}, and stochastic optimization \cite{NPP2025a,PP2024}. We highlight the relevance of our work within this broader effort throughout.

\subsection*{Preliminaries}

We assume only basic facts from probability, background for which can be found in any standard text (e.g.\ \cite{Kle2020,shiryaev:book}, which also include an exposition of the various well-known zero-one laws and their applications, providing further context for the results we present). 

What is not entirely standard is our general tendency to represent events using logical notation. Fixing some probability space $(\Omega,\mathcal{F},\PP)$, we regard a logical formula $B(x_1,\ldots,x_n)$ with parameters $x_1,\ldots,x_n$ as \emph{measurable} (and so can be assigned a probability) when $B(x_1,\ldots,x_n)$ also depends on $\omega\in \Omega$, and for any choice of parameters we have 
\[
\{\omega \mid B(x_1,\ldots,x_n)(\omega)\}\in \mathcal{F}. 
\]
As such, we freely use logical notation inside probability measures under the implicit assumption that any formulas that arise in this way can be appropriately represented as measurable formulas. For example, if some measurable $B(n)$ is parametrised by $n\in\NN$, then we can define $\exists n\, B(n):=\bigcup_{n=0}^\infty B(n)$ and $\forall n\, B(n):=\bigcap_{n=0}^\infty B(n)$ and thus talk about the probability of quantified formulas (note that the same cannot be so easily done for parameters that are not countable).

We will repeatedly use the following facts, which represent a simple reformulation of standard continuity properties of the probability measure, but isolate in a logical context the relationship between quantification inside and outside of the measure. We say that a measurable formula $B(n)$ parametrised by $n\in\NN$ is monotone decreasing (in $n$) if $B(n)\supseteq B(n+1)$ for all $n\in\NN$, and monotone increasing if $B(n)\subseteq B(n+1)$ for all $n\in\NN$, in both cases with $B(n)$ viewed as a subset of $\Omega$.
\begin{lemma}
\label{lem:continuity}
Let $p\in [0,1]$ and suppose that $B(n)$ is monotone decreasing in $n\in\NN$. Then
\begin{enumerate}[(i)]

	\item $\PP(\forall n\, B(n))\geq p\iff \forall n\, (\PP(B(n))\geq p)$.\smallskip
	
	\item $\PP(\forall n\, B(n))\leq p\iff \forall \lambda>0\, \exists n\, (\PP(B(n))<p+\lambda)$.

\end{enumerate}
On the other hand, if $B(n)$ is monotone increasing then
\begin{enumerate}[(i)]

	\item[(iii)] $\PP(\exists n\, B(n))\leq p\iff \forall n\, (\PP(B(n))\leq p)$.\smallskip
	
	\item[(iv)] $\PP(\exists n\, B(n))\geq p\iff \forall \lambda>0\, \exists n\, (\PP(B(n))> p-\lambda)$.

\end{enumerate}
\end{lemma}
This result is already given as \cite[Lemma 3.1]{neri-powell:25:martingale}, where a proof can be found.

\section{A zero-one law}
\label{sec:zeroone}

We consider a family $(B_{n,k})_{n,k\in\NN}$ of events indexed by a pair of natural numbers. A possible semantic interpretation is that $B_{n,k}\in \sigma(A_n,A_{n+1},\ldots,A_{k-1},A_k)$ for some sequence of independent events $(A_n)$, but we represent (and generalise) this syntactically by assuming that the following closure property is satisfied:
\[
B_{n,m}\cup B_{l,k}\subseteq B_{n,k} \mbox{ for all $n\leq m<l\leq k$},
\label{closure}\tag{P1}
\]
along with the following abstract condition, which generalises independence:
\[
\PP\left(\exists k\geq n\,  B_{n,k}\right)<1\implies \sum_{i=0}^\infty \PP\left(B_{a_i,b_i}\right)<\infty 
\label{ind}\tag{P2}
\]
for all $n\in\NN$ and sequences $n\leq a_0\leq b_0<a_1\leq b_1<\ldots$. In particular, it will be shown in Proposition \ref{prop:zeroone:ind} below that we can replace \eqref{ind} with the more concrete independence condition:
\[
\mbox{$B_{n,m}$ and $B_{l,k}$ are independent for all $n\leq m<l\leq k$.}
\label{pairwise}\tag{P3}
\]
For convenience we will assume that $B_{n,k}=\emptyset$ for $k<n$.
\begin{theorem}[Zero-one law]
\label{thm:zeroone}
Suppose that $(B_{n,k})$ satisfies \eqref{closure} and \eqref{ind}. Define
\[
B_\infty:=\forall n\, \exists k\geq n\, B_{n,k}=\bigcap_{n=0}^\infty\bigcup_{k=n}^\infty B_{n,k}.
\]
Then $\PP(B_\infty)\in \{0,1\}$.
\end{theorem}
Under the aforementioned semantic interpretation, Theorem \ref{thm:zeroone} follows directly from the Kolmogorov zero-one law, as $B_\infty$ would clearly form a tail event. However, Theorem \ref{thm:zeroone} should ultimately be viewed as a new zero-one law for tail events that enjoy a particularly nice syntactic structure, generalising the Borel-Cantelli lemmas (cf. Example \ref{ex:bc}) to the point where they can also be applied in settings where the Kolmogorov zero-one law would normally be used. Indeed, as we will see in Section \ref{sec:examples}, our zero-one law is rich enough to cover a range of interesting scenarios.

\begin{proof}[Proof of Theorem \ref{thm:zeroone}]
Using the notation $B_n:=\bigcup_{k=n}^\infty B_{n,k}$, we first note that \eqref{closure} implies that $B_{n,k}\supseteq B_{n+1,k}$ and $B_{n,k}\subseteq B_{n,k+1}$ for all $n,k\in\NN$, and therefore
\begin{itemize}
	\item $B_\infty=\forall n\, \exists j\, B_{n,n+j}$ where $\exists j\, B_{n,n+j}$ is monotone decreasing in $n$, and
	
	\item $B_n=\exists j\, B_{n,n+j}$ where $B_{n,n+j}$ is monotone increasing in $j$.
\end{itemize}
Therefore if $\PP(B_\infty)=\PP(\forall n\, \exists j\, B_{n,n+j})<1$, by Lemma \ref{lem:continuity} (ii) there exists some $N\in\NN$ such that 
\[
\PP\left(\exists k\geq N\,  B_{N,k}\right)<1
\]
and therefore we have
\begin{equation}
\label{finitesum}
\sum_{i=0}^\infty \PP\left(B_{a_i,b_i}\right)<\infty 
\end{equation}
for all $N\leq a_0\leq b_0<a_1\leq b_1<\ldots$. Suppose for contradiction that $\PP(B_\infty)>0$. Then there exists some $\varepsilon>0$ such that $\PP(B_\infty)\geq 2\varepsilon$, and by Lemma \ref{lem:continuity} (i) followed by (iv) we have
\[
\forall n\exists k\geq n\left(\PP\left(B_{n,k}\right)\geq \varepsilon\right).
\]
Define the sequences $(a_i),(b_i)$ by $a_0:=N$ and $a_{i+1}:=b_i+1$ where $b_i\geq a_i$ is such that $\PP\left(B_{a_i,b_i}\right)\geq \varepsilon$. Then \eqref{finitesum} is clearly violated for this choice.
\end{proof}
The following result justifies our characterisation of \eqref{ind} as an abstract independence condition.
\begin{proposition}
\label{prop:zeroone:ind}
Theorem \ref{thm:zeroone} holds with \eqref{ind} replaced by \eqref{pairwise}.
\end{proposition}

\begin{proof}
To see that \eqref{ind} follows from \eqref{pairwise} in the presence of \eqref{closure}, observe that for any $n\leq a_0\leq b_0<a_1\leq b_1<\ldots<a_j\leq b_j$, using \eqref{closure} and \eqref{pairwise} we have
\[
\PP\left(\bar{B}_{a_i,b_j}\right)\leq \PP\left(\bar{B}_{a_i,b_i}\cap \bar{B}_{a_{i+1},b_j}\right)=\PP\left(\bar{B}_{a_i,b_i}\right)\PP\left(\bar{B}_{a_{i+1},b_j}\right)
\]
for all $0\leq i<j$, and thus by induction
\[
\PP\left(\bar{B}_{a_0,b_j}\right)\leq \prod_{i=0}^j \PP\left(\bar{B}_{a_i,b_i}\right)=\prod_{i=0}^j\left(1-\PP\left(B_{a_i,b_i}\right)\right).
\]
Using the inequality $1+x\leq \exp(x)$ for all $x\in \RR$, we see that
\[
\PP\left(\bar{B}_{a_0,b_j}\right)\leq \prod_{i=0}^j\exp\left(-\PP\left(B_{a_i,b_i}\right)\right)=\exp\left(-\sum_{i=0}^j \PP\left(B_{a_i,b_i}\right)\right)
\]
and thus for arbitrary $j\in\NN$ we have
\[
\PP(\forall k\geq n\, \bar{B}_{n,k})\leq \PP(\bar{B}_{n,b_j})\leq \PP(\bar{B}_{a_0,b_j})\leq \exp\left(-\sum_{i=0}^j \PP\left(B_{a_i,b_i}\right)\right)
\]
and therefore if $\sum_{i=0}^\infty \PP\left(B_{a_i,b_i}\right)=\infty$ then $\PP(\forall k\geq n\, \bar{B}_{n,k})=0$, or in other words, $\PP(\exists k\geq n\, B_{n,k})=1$, and so \eqref{ind} has been established.
\end{proof} 
\begin{example}
\label{ex:bc}
Given a sequence of events $(A_j)_{j\in\NN}$, the family $B_{n,k}:=\bigcup_{j=n}^k A_j$ satisfies \eqref{closure}, and in this instance we have
\[
B_\infty=\forall n\, \exists k\geq n\, A_k=\{\text{$A_j$ occurs infinitely often}\}.
\]
If, furthermore, the $(A_j)$ are mutually independent, then \eqref{pairwise} holds and thus by Theorem \ref{thm:zeroone} via Proposition \ref{prop:zeroone:ind} we have 
\[
\PP\left(\text{$A_j$ occurs infinitely often}\right)\in \{0,1\}.
\]
In this sense, Theorem \ref{thm:zeroone} contains the first and second Borel-Cantelli lemmas.
\end{example}

\section{A finitisation of Theorem \ref{thm:zeroone}}
\label{sec:finitisation}

We now motivate, prove, and discuss our quantitative zero-one law, which represents a finitary version of Theorem \ref{thm:zeroone} in the sense of Tao \cite{tao:07:softanalysis}. Informally speaking, this forms an `approximate' zero-one law which asserts that an approximation of the event $B_\infty$ in terms of finite unions and intersections occurs with probability close to zero or one. This finitisation is achieved by combining continuity arguments that allow us to draw quantifiers outside of the probability measure, with G\"odel's (classical, monotone) Dialectica interpretation, which provides us with a recipe for Herbrandising the resulting formula in a way that explicit quantitative bounds can be guaranteed.

That the Dialectica interpretation can be used to systematically derive finitisations of infinitary analytic statements is well-known \cite{kohlenbach:08:book} (see also \cite{gaspar-kohlenbach:10:pigeonhole} for a particularly interesting case study in this regard), though the general applicability of Dialectica in the extremely subtle setting of probability theory has only recently begun to be explained from a proof-theoretic perspective in \cite{neri-pischke:pp:formal}.

However, we present our finitary zero-one law without making explicit the underlying logical methodology. In particular, we do not give the definition of the Dialectica interpretation, nor exact details on how it is applied: Neither of these is necessary to understand or prove the main results that follow. Instead we aim to provide just enough intuition that the reader not familiar with proof mining or Dialectica can see how the finitary version relates to the original zero-one law.\\

To begin the process of finitization, our first step is to formulate the main components of Theorem \ref{thm:zeroone} in terms of finite unions and intersections of events, bringing quantifiers from inside to outside the probability measures. 
\begin{lemma}
\label{lem:contents}
Whenever $(B_{n,k})$ satisfies \eqref{closure}, the statement \eqref{ind} is equivalent to the property that
	\begin{equation}
\exists \lambda\in (0,1)\, \forall k\geq n\left(\PP(B_{n,k})\leq 1-\lambda\right)\implies \exists x>0\, \forall t\left(\sum_{i=0}^t \PP\left(B_{a_i,b_i}\right)<x\right),
	\label{ind:quant}
	\end{equation}
for all $n\in\NN$ and $n\leq a_0\leq b_1<a_1\leq b_1<\ldots$, and the statement $\PP(B_\infty)\in \{0,1\}$ is equivalent to the formula
\begin{equation}
\forall\varepsilon\in (0,1)\, \exists n\, \forall m\left(\PP(B_{n,m})<\varepsilon\right)\vee \forall \lambda\in (0,1)\, \forall r\, \exists k\left(\PP(B_{r,k})>1-\lambda\right).
\label{ZO}
\end{equation}
\end{lemma}
\begin{proof}
Repeated applications of Lemma \ref{lem:continuity}, as in the proof of Theorem \ref{thm:zeroone}. The first part follows from observing that
\begin{align*}
\PP(\exists k\geq n\, B_{n,k})<1&\iff \exists \lambda\in (0,1)\, \left(\PP(\exists j\, B_{n,n+j})\leq 1-\lambda\right)\\
&\iff \exists \lambda\in (0,1)\,\forall j\, \left(\PP(B_{n,n+j})\leq 1-\lambda\right)\\
&\iff \exists \lambda\in (0,1)\,\forall k\geq n \left(\PP(B_{n,k})\leq 1-\lambda\right).
\end{align*}
For the second, we first note that
\begin{align*}
\PP(B_\infty)=0&\iff \PP(\forall n\, \exists j\, B_{n,n+j})=0\\
&\iff \forall\varepsilon\in (0,1)\, \exists n\left(\PP(\exists j\, B_{n,n+j})<\varepsilon\right),
\end{align*}
and then
\begin{align*}
&\forall\varepsilon\in (0,1)\, \exists n\left(\PP(\exists j\, B_{n,n+j})<\varepsilon\right)\\
&\iff \forall\varepsilon\in (0,1)\, \exists n,\mu\in (0,1)\left(\PP(\exists j\, B_{n,n+j})\leq\varepsilon-\mu\right)\\
&\iff \forall\varepsilon\in (0,1)\, \exists n,\mu\in (0,1)\, \forall j\left(\PP(B_{n,n+j})\leq\varepsilon-\mu\right)\\
&\implies \forall\varepsilon\in (0,1)\, \exists n\, \forall m\left(\PP(B_{n,m})<\varepsilon\right)
\end{align*}
(where for the last step we recall  that $B_{n,m}=\emptyset$ for $m<n$), and
\begin{align*}
&\forall\varepsilon\in (0,1)\, \exists n\, \forall m\left(\PP(B_{n,m})<\varepsilon\right)\\
&\implies \forall\varepsilon\in (0,1)\, \exists n\left(\PP(\exists j\, B_{n,n+j})\leq\varepsilon\right)\\
&\iff \forall\varepsilon\in (0,1)\, \exists n\left(\PP(\exists j\, B_{n,n+j})< \varepsilon\right).
\end{align*}
Finally, we have
\begin{align*}
\PP(B_\infty)=1&\iff \PP(\forall r\, \exists j\, B_{r,r+j})=1\\
&\iff \forall r\left(\PP(\exists j\, B_{r,r+j})=1\right)\\
&\iff \forall r\, \forall \lambda\in (0,1)\, \exists j \left(\PP(B_{r,r+j})>1-\lambda\right)\\
&\iff \forall r\, \forall \lambda\in (0,1)\, \exists k \left(\PP(B_{r,k})>1-\lambda\right),
\end{align*}
and the result follows.
\end{proof}
Our second step is to give a computational interpretation to \eqref{ind:quant} and \eqref{ZO}, bringing the quantifiers, already removed from the probability measures, to the front of the formula, and here we make use of the Dialectica interpretation. We do not give full details, as the resulting finitary version will in any case be given below, but rather focus on motivating the interpretation of the (more complex) formula \eqref{ZO}. Here we make use of Dialectica in the form of the Shoenfield variant \cite{shoenfield:67:book}, whereby \eqref{ZO} becomes first
\[
\forall\varepsilon\in (0,1)\, \forall g\, \exists n\left(\PP(B_{n,g(n)})<\varepsilon\right)\vee \forall \lambda\in (0,1)\, \forall r\, \exists k\left(\PP(B_{r,k})>1-\lambda\right)
\]
and finally
\begin{equation}
\forall \varepsilon,\lambda\in (0,1)\, \forall r\, \forall g:\NN\to\NN\, \exists k,n\left(\PP(B_{n,g(n)})< \varepsilon\vee \PP(B_{r,k})> 1-\lambda\right)
\label{mZO}.
\end{equation}
Despite the formal apparatus invoked in the passage from \eqref{ZO} to \eqref{mZO}, the equivalence of these two formulas (over classical logic) can be proven directly:
\begin{lemma}
\label{lem:mzo}
Statements \eqref{ZO} and \eqref{mZO} are equivalent.
\end{lemma}
\begin{proof}
The negation of \eqref{ZO} is equivalent to
\[
\exists\varepsilon\in (0,1)\, \forall n\, \exists m\left(\PP(B_{n,m})\geq\varepsilon\right)\wedge \exists \lambda\in (0,1)\, \exists r\, \forall k\left(\PP(B_{r,k})\leq 1-\lambda\right).
\]
Bringing the quantifiers to the outside, and introducing a function $g$ for witnessing $m$ in terms of $n$, we obtain the equivalent statement
\[
\exists \varepsilon,\lambda\in (0,1)\, \exists r\, \exists g:\NN\to\NN\, \forall k,n\left(\PP(B_{n,g(n)})\geq \varepsilon\wedge \PP(B_{r,k})\leq 1-\lambda\right)
\]
and negating once more we arrive at \eqref{mZO}.
\end{proof} 
The intention here is that \eqref{mZO} bears the same relationship to \eqref{ZO} as metastable convergence (in the sense of Tao \cite{tao:07:softanalysis}):
\begin{equation}
\forall \varepsilon>0\, \forall g:\NN\to\NN\, \exists n\, \forall i,j\in [n;g(n)](|x_i-x_j|<\varepsilon)
\label{metastable}
\end{equation}
bears to the usual Cauchy property:
\[
\forall \varepsilon>0\, \exists n\, \forall i,j\geq n(|x_i-x_j|<\varepsilon).
\]
It is a well-known and widely exploited phenomenon that while in many situations, computable rates of Cauchy convergence do not exist, it is usually possible to extract bounds for $\exists n$ in \eqref{metastable} which are typically simple and highly uniform (see \cite{kohlenbach:19:nonlinear:icm} for a reasonably recent list of examples). In the remainder of the section we show that the same is true for Theorem \ref{thm:zeroone}. 

\begin{remark}[For proof theorists]
This comparison to the nonstochastic setting must be carefully qualified. In contrast to ordinary, nonstochastic metastability, our route from the initial zero-one law to its metastable version involves the additional step of taking quantifiers out of the probability measure (Lemma \ref{lem:contents}), using the equivalences of presented in Lemma \ref{lem:continuity}. While the latter are completely elementary within the context of ordinary probability theory, from a logical point of view they can only be done in a system capable of reasoning about countable unions, which turn out to be set-theoretically subtle. For example, an investigation into the Lesbegue measure in the context of higher-order reverse mathematics utilises a uniform variant of arithmetical comprehension in the form of Feferman's $\mu$ functional \cite{kreuzer:15:measure}. Similarly, in the recent formal system for probability given in \cite{neri-pischke:pp:formal}, in order to guarantee a tame treatment, infinite unions are formulated intensionally and with their charaterising axiom given in rule form. As such, there is currently no general, theoretical justification for the notion that a metastable principle arrived at by using countable unions in a fundamental way can be given a computational interpretation of low complexity\footnote{A systematic approach towards the computational interpretation of probabilistic statements in light of recent logical metatheorems \cite{neri-pischke:pp:formal} will be provided in a forthcoming paper by Neri, Oliva and Pischke.}. That this is indeed possible for us, whereby we seemingly avoid the potential blowup in complexity triggered by invoking continuity principles of the measure (as is also the case in e.g.\ martingale convergence \cite{neri-powell:25:martingale}), is due to the fact that the \emph{proof} of Theorem \ref{thm:zeroone} is fundamentally uniform in nature i.e.\ can be formulated with the quantifiers taken outside of the measure and without invoking any strong set-theoretic principles. Though our finitary zero-one law can be understood without reference to any of these issues, we mention them simply to stress that a obtaining a thorough, proof-theoretic understanding of probability theory is an ongoing effort, and our finitary zero-one law benefits from uniformities that are not necessarily guaranteed for general theorems of probability (but which are nevertheless often present).
\end{remark}

The quantitative results given by Theorem \ref{thm:main} below and its corollaries represent a procedure for converting witnesses for the Dialectica interpretation of the main assumption \eqref{ind:quant} into bounds for witnesses of the metastable zero-one property \eqref{mZO}, though we now state and prove the result without any reference to proof theoretic methods. The fact that we ask for \emph{bounds} rather than exact witnesses allows us to formulate our quantitative results in a simple and highly uniform way, a phenomenon that can made precise through the monotone variant of the Dialectica (see \cite{kohlenbach:08:book}, and \cite{neri-pischke:pp:formal} for the specific case of probability). The results are finitary in that, while being equivalent to Theorem \ref{thm:zeroone} (cf.\ Remark \ref{rem:equiv} below), both the main assumption \eqref{ind} and the conclusion $\PP(B_\infty)\in \{0,1\}$ now only involve a finite part of the input $(B_{n,k})$. Furthermore, the assumption \eqref{closure} is not initially needed and only comes back into play when we replace the \eqref{ind} with a concrete independence condition. Much of the technical work has already been done in arriving at the correct statement of the theorem, and as a result the proof is completely elementary. 

\begin{theorem}[Finitary zero-one law]
\label{thm:main}
Let $(B_{n,k})$ be arbitrary. Fix $\varepsilon,\lambda\in (0,1)$, $r\in\NN$ and $g:\NN\to\NN$ with $g(k)\geq k$ for all $k\in\NN$, and define $a_0^{r,g}\leq b_0^{r,g}<a_1^{r,g}\leq b_1^{r,g}\leq \ldots$ by
\[
a^{r,g}_i:=\tilde g^{(i)}(r) \ \ \ \mbox{and} \ \ \ b^{r,g}_i:=g(a^{r,g}_i)
\]
for $\tilde g(j):=g(j)+1$. Then whenever $x>0$ and $s\geq r$ are such that
\[
\PP(B_{r,s})\leq 1-\lambda\implies \sum_{i=0}^{\lfloor x/\varepsilon\rfloor} \PP\left(B_{a^{r,g}_i,b^{r,g}_i}\right)<x,
	\label{ind:finitary}\tag{P2$'$}
\]
either 
\[
\exists n\leq \tilde g^{(\lfloor x/\varepsilon\rfloor)}(r)\left(\PP\left(B_{n,g(n)}\right)<\varepsilon\right) \ \ \ \mbox{or} \ \ \ 1-\lambda<\PP\left(B_{r,s}\right).
\]
\end{theorem}
\begin{proof}
Whenever $\PP\left(B_{r,s}\right)\leq 1-\lambda$, there exists some $i\leq \lfloor x/\varepsilon\rfloor$ such that $\PP\left(B_{a^{r,g}_i,b^{r,g}_i}\right)<\varepsilon$, since if this were not true, then we would have
\[
\sum_{i=0}^{\lfloor x/\varepsilon\rfloor} \PP\left(B_{a^{r,g}_i,b^{r,g}_i}\right)\geq \sum_{i=0}^{\lfloor x/\varepsilon\rfloor} \varepsilon=(\lfloor x/\varepsilon\rfloor+1)\varepsilon\geq x
\]
contradicting \eqref{ind:finitary}. The result then follows, since $b^{r,g}_i=g(a^{r,g}_i)$ and $a^{r,g}_i\leq a^{r,g}_{\lfloor x/\varepsilon\rfloor}=\tilde g^{(\lfloor x/\varepsilon\rfloor)}(r)$. 
\end{proof}

The following variant of our finitary zero-one law now assumes that we have a (uniform) witness for the Dialectica interpretation of \eqref{ind:quant}, allowing us to produce concrete witnesses for the conclusion for any parameters.

\begin{corollary}
\label{cor:implication}
Let $(B_{n,k})$ be an arbitrary, and suppose that $\rho:\NN\times (0,1)\to (0,\infty)$ and $\sigma:\NN^2\times (0,1)\to \NN$ with $\sigma(N,n,\lambda)\geq n$ for all $n\in\NN$ satisfy
\[
\PP\left(B_{n,\sigma(N,n,\lambda)}\right)\leq 1-\lambda\implies \forall t\left(b_t\leq N\implies \sum_{i=0}^t\PP\left(B_{a_i,b_i}\right)<\rho(n,\lambda)\right)
\label{ind:dial}\tag{P2$''$}
\]
for all $n,N\in\NN$ and $\lambda\in (0,1)$, and uniformly over sequences $n\leq a_0\leq b_0<a_1\leq b_1<\ldots$. Then for any $\varepsilon,\lambda\in (0,1)$, $r\in\NN$ and $g:\NN\to\NN$ with $g(k)\geq k$ for all $k\in\NN$, either 
\[
\exists n\leq \tilde g^{(\lfloor \rho(r,\lambda)/\varepsilon\rfloor)}(r)\left(\PP\left(B_{n,g(n)}\right)\right)<\varepsilon \ \ \ \mbox{or} \ \ \ 1-\lambda<\PP\left(B_{r,\sigma(b^{r,g}_{\lfloor \rho(r,\lambda)/\varepsilon\rfloor},r,\lambda)}\right).
\]
for $(b^{r,g}_n)$ defined as in Theorem \ref{thm:main}.
\end{corollary}
\begin{proof}
Directly from Theorem \ref{thm:main}, setting $x:=\rho(r,\lambda)$ and $s:=\sigma(N,r,\lambda)$ for $N:=b^{r,g}_{\lfloor x/\varepsilon\rfloor}=b^{r,g}_{\lfloor\rho(r,\lambda)/\varepsilon\rfloor}$.
\end{proof}

Finally, we provide a finitary version of Proposition \ref{prop:zeroone:ind}, whose premises are trivially satisfied when \eqref{closure} and \eqref{pairwise} hold universally, but which makes explicit that these properties are in fact only required to hold on finite ranges for any fixed degree of accuracy.

\begin{corollary}
\label{cor:independent}
Let $(B_{n,k})$ be arbitrary. Fix $\varepsilon,\lambda\in (0,1)$, $r\in\NN$ and $g:\NN\to\NN$ with $g(k)\geq k$ for all $k\in\NN$, and define $r\leq a_0^{r,g}\leq b_0^{r,g}<a_1^{r,g}\leq b_1^{r,g}\leq \ldots$ as in Theorem \ref{thm:main}. Then whenever
\begin{itemize}

	\item $B_{n,m}\cup B_{l,k}\subseteq B_{n,k}$ and
	
	\item $B_{n,m}$ and $B_{l,k}$ are independent

\end{itemize}
for all $r\leq n\leq m<l\leq k\leq s:=b^{r,g}_{\lfloor \log(1/\lambda)/\varepsilon\rfloor+1}$, either
\[
\exists n\leq \tilde g^{(\lfloor \log(1/\lambda)/\varepsilon\rfloor+1)}(r)\left(\PP\left(B_{n,g(n)}\right)<\varepsilon\right) \ \ \ \mbox{or} \ \ \ 1-\lambda<\PP\left(B_{r,s}\right).
\]
\end{corollary}

\begin{proof}
Fix $\varepsilon,\lambda\in (0,1)$, $r\in\NN$ and $g:\NN\to\NN$. For $j:=\lfloor \log(1/\lambda)/\varepsilon\rfloor+1$, if $\PP\left({B}_{r,b^{r,g}_j}\right)\leq 1-\lambda$ then reasoning exactly as in the proof of Proposition \ref{prop:zeroone:ind}, noting that \eqref{closure} and \eqref{pairwise} hold on the relevant range, we have
\[
\lambda\leq \PP\left(\bar{B}_{r,b^{r,g}_j}\right)\leq \PP\left(\bar{B}_{a^{r,g}_0,b^{r,g}_j}\right)\leq \exp\left(-\sum_{i=0}^j \PP\left(B_{a^{r,g}_i,b^{r,g}_i}\right)\right)
\]
and thus setting $x:=\log(1/\lambda)+\varepsilon$ we have
\[
\sum_{i=0}^{\lfloor x/\varepsilon\rfloor} \PP\left(B_{a^{r,g}_i,b^{r,g}_i}\right)=\sum_{i=0}^j \PP\left(B_{a^{r,g}_i,b^{r,g}_i}\right)\leq \log(1/\lambda)<x,
\]
and so we have shown that \eqref{ind:finitary} holds for this choice of $x$ and $s:=b^{r,g}_j$, from which the result follows.
\end{proof}

\begin{example}
\label{ex:bc:finitary}
Continuing Example \ref{ex:bc}, Corollary \ref{cor:independent} gives us a finitary analogue of the fact that
\[
\PP\left(\text{$A_j$ occurs infinitely often}\right)\in \{0,1\}
\]
whenever $(A_j)$ are mutually independent: Specifically, for any $\varepsilon,\lambda\in (0,1)$, $r\in\NN$ and $g:\NN\to\NN$ with $g(k)\geq k$ for all $k\in\NN$, whenever the events
\[
\{A_r,A_{r+1},\ldots,A_{s-1},A_s\}
\]
are mutually independent for $s:=b^g_{\lfloor \log(1/\lambda)/\varepsilon\rfloor+1}$ as defined in Corollary \ref{cor:independent}, either
\[
\exists n\leq \tilde g^{(\lfloor \log(1/\lambda)/\varepsilon\rfloor+1)}(r)\left(\PP\left(\bigcup_{j=n}^{g(n)} A_j\right)<\varepsilon\right) \ \ \ \mbox{or} \ \ \ \PP\left(\bigcup_{j=r}^s A_j\right)>1-\lambda.
\]
This represents a quantitative version of the zero-one law arising from the first and second Borel-Cantelli lemmas, which were individually analysed in \cite[Theorems 2.1 and 2.2]{arthan-oliva:21:borel-cantelli}. However, metastability is only a feature of our approach, arising from the instance of law of excluded middle required when combining the two Borel-Cantelli lemmas in this way. 
\end{example}

\begin{remark}[Deriving Theorem \ref{thm:zeroone} from Theorem \ref{thm:main}]
\label{rem:equiv}
Suppose that \eqref{closure} and \eqref{ind} are satisfied. It suffices to show that \eqref{mZO} holds, which by Lemmas \ref{lem:contents} and \ref{lem:mzo} is equivalent to $\PP(B_\infty)\in \{0,1\}$. To that end, fix $\varepsilon,\lambda\in (0,1)$, $r\in\NN$ and $g:\NN\to\NN$, and let $(a^{r,g}_n)$ and $(b^{r,g}_n)$ be defined as in Theorem \ref{thm:main}. By Lemma \ref{lem:contents}, \eqref{ind} is equivalent to \eqref{ind:quant}, which in turn implies that there exists some $x>0$ such that
\[
\forall k\geq r\left(\PP(B_{r,k})\leq 1-\lambda\right)\implies \forall t\left(\sum_{i=0}^t \PP(B_{a^{r,g}_i,b^{r,g}_i})<x\right).
\]
Setting $t:=\lfloor x/\varepsilon\rfloor$, it follows (again, classically) that there exists some $s\geq r$ such that \eqref{ind:finitary} holds. We then easily infer that 
\[
\exists k,n\left(\PP(B_{n,g(n)})<\varepsilon\vee \PP(B_{r,k})>1-\lambda\right)
\]
using Theorem \ref{thm:main}, and since $\varepsilon,\lambda,r$ and $g$ are arbitrary, this establishes \eqref{mZO}.
\end{remark}

\begin{remark}[Using the finitary zero-one law]
The disjunction in the conclusion of our finitary zero-one laws is not in general decidable, and so in that sense our main result is not fully constructive. However, zero-one laws are often used in the form of an implication $\PP(B_\infty)<1\implies \PP(B_\infty)=0$, and here our finitisation via the Shoenfield interpretation coincides with the Dialectica interpretation of the implication
\[
\exists \lambda\in (0,1)\, \exists r\, \forall k\left(\PP(B_{r,k})\leq 1-\lambda\right)\implies \forall \varepsilon\in (0,1)\, \forall g\,\exists n\left(\PP(B_{n,g(n)})<\varepsilon\right),
\]
and therefore in such cases one can use it to provide a quantitative route from $\PP(B_\infty)<1$ to $\PP(B_\infty)=0$. To be more precise, in situations where we take as premise that $\PP(B_\infty)<1$, and this comes equipped with concrete witnesses $\lambda>0$ and $r\in\NN$ satisfying
\[
\PP\left(\exists k\geq r\, B_{r,k}\right)\leq 1-\lambda
\label{eqn:right}\tag{R}
\]
we can then obtain a rate of metastability rate for the statement $\PP(B_\infty)=0$, in the sense of a concrete bound on a witness $n$ for
\[
\forall \varepsilon\in (0,1)\, \forall g\,\exists n\left(\PP(B_{n,g(n)})<\varepsilon\right).
\]
We will give two example where the finitary zero-one law is used in this way, relating to Kolmgorov's three-series theorem Theorem \ref{thm:threeseries}) and bond percolation (Corollary \ref{percolation2}), and we conjecture that there are many other cases where our finitary zero-one law can be used to generate bounds on metastable statements in the probability literature. 
\end{remark}

\begin{remark}[The necessity of metastability]
\label{rem:specker}
It is natural to ask whether the indirect computational interpretation provided by metastability is necessary. We show that even in the case that \eqref{eqn:right} is satisfied for concrete $\lambda,r$, a \emph{direct} computational interpretation of $\PP(B_\infty)=0$ in the sense of a computable function $\phi:(0,1)\to \NN$ satisfying 
\[
\forall \varepsilon>0\, \forall k\left(\PP(B_{\phi(\varepsilon),k})<\varepsilon\right)
\label{S}\tag{S}
\]
is not generally possible. Let $(X_n)$ be a sequence of independent random variables all uniformly distributed on $[0,1]$, and $(q_n)\subset (0,1)$ a monotonically increasing sequence of rationals that converge to a noncomputable limit $q$ (the existence of such sequences is a standard fact of computable analysis going back to Specker \cite{specker:49:sequence}). Now defining $(B_{n,k})$ as in Example \ref{ex:bc} for $A_n$ the event $X_n\in [q_n,q_{n+1}]$, we have that the $(A_n)$ are mutually independent by independence of $(X_n)$, and furthermore
\begin{align*}
\PP\left(\bigcup_{k=0}^\infty B_{0,k}\right)=\PP\left(\bigcup_{k=0}^\infty A_{k}\right)&\leq \lim_{m\to\infty}\sum_{k=0}^m \PP(A_k)=\lim_{m\to\infty}\sum_{k=0}^m (q_{k+1}-q_k)\\
&=\lim_{m\to\infty}(q_{m+1}-q_0)= q-q_0\leq 1-q_0
\end{align*}
and thus \eqref{eqn:right} holds for $r:=0$ and $\lambda:=q_0$. However, were \eqref{S} to hold for some computable $\phi$, then by similar reasoning to the proof of Proposition \ref{prop:zeroone:ind} it follows that
\[
\forall \varepsilon\in (0,1)\, \left(\sum_{i=\phi(\varepsilon)}^\infty \PP(A_i)\leq-\log(1-\varepsilon)\right),
\]
and so setting $\psi(\delta):=\phi(1-\exp(-\delta))$, for any $\delta>0$ and $n\geq \psi(\delta)$ we have
\[
q-q_n=\sum_{i=n}^\infty \PP(A_i)\leq \sum_{i=\psi(\delta)}^\infty \PP(A_i)\leq \delta, 
\]
and thus $\psi$ is a computable rate of convergence for $(q_n)$, a contradiction.
\end{remark}

\section{Examples}
\label{sec:examples}

We now provide three short case studies where we illustrate how the results given in the previous section can be utilised to provide new quantitative, finitary theorems in different areas of probability, demonstrating that our syntactic zero-one law and its finitary counterparts are broadly applicable. Our case studies focus, respectively, on: (a) extensions to Example \ref{ex:bc} -- building on recent work of Arthan and Oliva \cite{arthan-oliva:21:borel-cantelli}; (b) metastable almost sure convergence, introduced from a logical perspective and for general measures by Avigad et al. in \cite{avigad-gerhardy-towsner:10:local} and extensively studied in recent years \cite{neri:25:kronecker,neri-pischke:pp:formal,neri-pischke-powell:25:learnability,neri-powell:pp:rs,neri-powell:25:martingale}, where we use our zero-one law to provide a quantitative version of the Kolmogorov three-series theorem; and (c) percolation theory, considered from the perspective of applied proof theory for the very first time in this paper.

\subsection{The Erd\H{o}s-R\'enyi theorem}
\label{sec:erdosrenyi}

We first consider a zero-one law based on the so-called Erd\H{o}s-R\'enyi theorem \cite{erdos-renyi:59:cantor}, inspired by the quantitative study of the latter in \cite{arthan-oliva:21:borel-cantelli}. This represents both an expansion of their \cite[Theorem 3.3]{arthan-oliva:21:borel-cantelli}, and a strengthening of our Example \ref{ex:bc:finitary}, where mutual independence of the $(A_j)$ is weakened to the following condition (here and in the remainder of this section we use the more compact notation $A_iA_j:=A_i\cap A_j$):
\[
\liminf_{n\to\infty} \frac{\sum_{i,j=0}^n \PP(A_iA_j)}{\left(\sum_{k=0}^n \PP(A_k)\right)^2}=1.
\label{ER}
\]
As such, this example demonstrates the versatility of our abstract independence condition \eqref{ind} and its quantitative counterpart \eqref{ind:dial}. We begin by restating \cite[Theorem 3.3]{arthan-oliva:21:borel-cantelli}, where now we modify their assumption $\sum_{n=0}^\infty \PP(A_i)=\infty$ by making explicit precisely how it is used for fixed parameters.
\begin{lemma}[Due to \cite{arthan-oliva:21:borel-cantelli}\footnote{Here we also provide a simplification of the proof of \cite[Theorem 3.3]{arthan-oliva:21:borel-cantelli} communicated to the authors by Paulo Oliva, which reduces the complexity of the resulting bound.}]
\label{lem:er}
Let $(A_n)$ be an arbitrary sequence of events, and suppose that $\phi:(0,1)\times\NN\to\NN$ satisfies 
\[
\forall \lambda\in (0,1),n\in\NN\left(\phi(\lambda,n)\geq n\wedge \frac{\sum_{i,j=0}^{\phi(\lambda,n)}\PP(A_iA_j)}{\left(\sum_{k=0}^{\phi(\lambda,n)}\PP(A_k)\right)^2}< 1+\lambda\right).
\]
Then for any $n,N\in\NN$, if 
\[
\sum_{i=0}^N \PP(A_i)\geq 2n
\]
then for all $\lambda\in(0,1)$
\[
\PP\left(\bigcup_{i=n}^{\phi(\lambda/4,N)} A_i\right)> 1-\lambda.
\]
\end{lemma}
\begin{proof}
Let $B_{n,k}:=\bigcup_{j=n}^k A_j$ as in Example \ref{ex:bc}, $X_n:=1_{A_n}$ and $Y_n:=\sum_{i=0}^n X_i$. Noting that
\[
\frac{\sum_{i,j=0}^{n}\PP(A_iA_j)}{\left(\sum_{k=0}^{n}\PP(A_k)\right)^2}=\frac{\EE\left[Y_{n}^2\right]}{\EE\left[Y_{n}\right]^2}=\frac{\var(Y_n)}{\EE\left[Y_{n}\right]^2}+1
\]
it follows from the defining property of $\phi$ that
\[
\forall \lambda,n\left(\phi(\lambda/4,n)\geq n\wedge \frac{\var(Y_{\phi(\lambda/4,n)})}{\EE\left[Y_{\phi(\lambda/4,n)}\right]^2}< \frac{\lambda}{4}\right).
\]
Fixing $n\in\NN$ and $\lambda\in (0,1)$, let $N\in\NN$ be such that $\sum_{i=0}^N \PP(A_i)\geq 2n$. Define $Y:=Y_{\phi(\lambda/4,N)}$, $\sigma^2:=\var(Y_{\phi(\lambda/4,N)})$ and $\mu:=\EE[Y]=\EE[Y_{\phi(\lambda/4,N)}]$. In particular we have $\sigma^2/\mu^2< \lambda/4$ and therefore by Chebyshev's inequality 
\[
\PP\left(Y\leq \frac{\mu}{2}\right)\leq \PP\left(|Y-\mu|\geq \frac{\mu}{2}\right)\leq \frac{4\sigma^2}{\mu^2}< \lambda
\]
and thus
\[
\PP\left(Y>\frac{\mu}{2}\right)> 1-\lambda.
\]
It therefore suffices to show that
\[
\left\{Y_{\phi(\lambda/4,N)}>\frac{\EE[Y_{\phi(\lambda/4,N)}]}{2}\right\}=\left\{Y>\frac{\mu}{2}\right\}\subseteq \bigcup_{i=n}^{\phi(\lambda/4,N)} A_i.
\]
To this end, first observe that since $\phi(\lambda/4,N)\geq N$ then
\[
\frac{\EE[Y_{\phi(\lambda/4,N)}]}{2}\geq \frac{\EE[Y_{N}]}{2}\geq \frac{1}{2}\sum_{i=0}^N \PP(A_i)\geq n.
\]
Therefore for any $\omega\in \left\{Y>\frac{\mu}{2}\right\}$, we have
\[
Y_{\phi(\lambda/4,N)}(\omega)=\sum_{i=0}^{\phi(\lambda/4,N)}X_i(\omega)\geq n+1,
\]
and therefore by the pigeonhole principle we have $X_i(\omega)=1$ for at least one of $\{n,n+1,\ldots,\phi(\lambda/4,N)\}$, and thus $\omega\in \bigcup_{i=n}^{\phi(\lambda/4,N)} A_i$.
\end{proof}

\begin{theorem}
\label{res:erdos:quantit}
Let $(A_n)$ be an arbitrary sequence of events, and let $\phi:(0,1)\times\NN\to\NN$ be as in Lemma \ref{lem:er}. Then for any $\varepsilon,\lambda\in (0,1)$, $r\in\NN$ and $g:\NN\to\NN$ with $g(k)\geq k$ for all $k\in\NN$, either 
\[
\exists n\leq \tilde g^{(\lfloor 2r/\varepsilon\rfloor)}(r)\,\PP\left(\bigcup_{j=n}^{g(n)}A_j\right)<\varepsilon \ \ \mbox{or} \ \ 1-\lambda<\PP\left(\bigcup_{j=r}^{\phi\left(\lambda/4,b^{r,g}_{\lfloor 2r/\varepsilon\rfloor}\right)}A_j\right).
\]
for $(b^{r,g}_n)$ defined as in Theorem \ref{thm:main}.
\end{theorem}

\begin{proof}
Let $B_{n,k}:=\bigcup_{j=n}^k A_j$ as in Example \ref{ex:bc}. It follows by Lemma \ref{lem:er} that
\begin{equation}
\forall n,N\in\NN\, \forall \lambda\in (0,1)\left(\PP\left(B_{n,\phi(\lambda/4,N)}\right)\leq 1-\lambda\implies\sum_{i=0}^N \PP(A_i)< 2n\right).
\label{paulo}
\end{equation}
We now claim that the independence condition \eqref{ind:dial} used in Corollary \ref{cor:implication} holds for $\rho(n,\lambda):=2n$ and $\sigma(N,n,\lambda):=\phi(\lambda/4,N)$. To see this, fix $n,N\in \NN$, $\lambda\in (0,1)$ and $n\leq a_0\leq b_0<a_1\leq b_1<\ldots$. Then from $\PP\left(B_{n,\sigma(N,n,\lambda)}\right)\leq 1-\lambda$, it follows from \eqref{paulo} that for any $t\in\NN$, whenever $b_t\leq N$ we have
\[
\sum_{i=0}^t \PP(B_{a_i,b_i})\leq \sum_{i=0}^t\sum_{j=a_i}^{b_i} \PP(A_j)\leq \sum_{j=a_0}^{b_t}\PP(A_j)\leq \sum_{j=0}^{N} \PP(A_j)<2n.
\]
That establishes the claim, and the result then follows from Corollary \ref{cor:implication}.
\end{proof}

\subsection{A criterion for almost sure convergence}
\label{sec:convergence}

In the last years, applied proof theory has seen an increased focus on metastable convergence theorems for \emph{stochastic processes}, encompassing in particular martingale convergence \cite{neri-powell:pp:rs,neri-powell:25:martingale} and laws of large numbers \cite{neri:25:kronecker}. Here, ordinary metastable convergence in the sense of \eqref{metastable} needs to be generalised to the stochastic setting, and a particularly useful notion is that of \emph{uniform metastable convergence}, first studied in a logical context in \cite{avigad-dean-rute:12:dominated,avigad-gerhardy-towsner:10:local}:
\begin{equation}
\forall \varepsilon>0\,\forall p\in\NN\, \forall g:\NN\to\NN\, \exists n\left(\PP\left(\exists i,j\in [n;g(n)](|X_i-X_j|\geq 2^{-p}\right)<\varepsilon\right),
\label{uniformmetastable}
\end{equation}
where here we use the notation $[n;k]:=\{n,n+1,\ldots,k-1,k\}$. We now show how our finitary zero-one law can be used to provide a quantitative almost-sure convergence criterion for series of random variables, capturing the qualitative idea that if $(X_n)$ are independent, then $\sum_{n=0}^\infty X_n$ converges with probability zero or one. In particular, if we can show that $\sum_{n=0}^\infty X_n$ converges with probability greater than zero, we can upgrade this to almost sure convergence, and our framework provides a way to do this quantitatively:
\begin{theorem}
\label{thm:conv}
Let $(X_n)$ be a sequence of independent random variables and define $S_n:=\sum_{i=0}^n X_i$. Suppose furthermore that $\phi:\NN\to \NN$ and $\lambda\in(0,1)$ are such that for all $p\in\NN$:
\begin{equation}
\PP\left(\exists i,j\geq \phi(p)(|S_i-S_j|\geq 2^{-p})\right)\leq 1-\lambda.
\label{notone}
\end{equation}
Then $(S_n)$ converges almost surely, where for any $\varepsilon>0$, $p\in\NN$ and $g:\NN\to\NN$ with $g(k)\geq k$ for all $k\in\NN$, there exists some $n\leq \tilde g^{(\lfloor\log(1/\lambda)/\varepsilon\rfloor+1)}(\phi(p))$ such that
\[
\PP\left(\exists i,j\in [n;g(n)](|S_i-S_j|\geq 2^{-p}\right)<\varepsilon.
\]
\end{theorem}
\begin{proof}
Fixing $p\in\NN$, we apply Corollary \ref{cor:independent} with
\[
B^p_{n,k}:=\exists i,j\in [n;k](|S_i-S_j|\geq 2^{-p}).
\]
It is clear that the requisite properties of $(B_{n,k})$ are satisfied in this case, and on any range, where for independence we note that $B^p_{n,k}$ depends only on the random variables $\{X_{n+1},\ldots,X_k\}$. Fixing also $\varepsilon\in (0,1)$ and $g:\NN\to\NN$, the result then follows by applying Corollary \ref{cor:independent} with $r:=\phi(p)$ and $\lambda$, for $\phi$ and $\lambda$ satisfying the premise of the theorem, which ensures that
\[
\PP(B_{\phi(p),s})\leq 1-\lambda
\]
for any $s\in\NN$.
\end{proof}
\begin{example}
\label{ex:criterion}
For a simple example where \eqref{notone} holds trivially but the best we can hope for are metastable convergence rates for $(S_n)$, let $(A_n)$ be defined just as in Remark \ref{rem:specker} i.e.\ $A_n$ denotes the event $U_n\in [q_n,q_{n+1}]$ for $(U_n)$ a sequence of independent random variables uniformly distributed on $[0,1]$ and $(q_n)\subset (0,1)$ a monotone increasing sequence converging to a noncomputable real $q$. Define $X_n:=I_{A_n}$. Then for any $p\in\NN$ we have
\begin{align*}
\PP\left(\exists i,j(|S_i-S_j|\geq 2^{-p})\right)&=\PP\left(\exists k(X_k=1)\right)=\PP\left(\exists k\, A_k\right)\leq \sum_{k=0}^\infty \PP(A_k)=q-q_0
\end{align*}
and thus \eqref{notone} holds for $\phi(p):=0$ and $\lambda:=q_0$. On the other hand, were we to have some computable $\varphi:\NN\times (0,1)\to \NN$ such that for all $p\in\NN$ and $\varepsilon\in(0,1)$:
\[
\PP\left(\exists i,j\geq\varphi(p,\varepsilon)(|S_i-S_j|\geq 2^{-p})\right)<\varepsilon,
\]
then in particular we would have for all $k\in\NN$ and $\varepsilon\in (0,\varepsilon)$
\begin{align*}
\PP\left(\exists i\in [\varphi(0,\varepsilon)+1;k]\, A_i\right)=\PP\left(\exists i,j\in [\varphi(0,\varepsilon);k](|S_i-S_j|\geq 1)\right)<\varepsilon
\end{align*}
and by exactly the same reasoning as Remark \ref{rem:specker}, it follows that
\[
\forall \varepsilon\in (0,1)\left(\sum_{i=\varphi(0,\varepsilon)+1}^\infty \PP(A_i)\leq -\log(1-\varepsilon)\right),
\]
which can then be used to construct a computable rate of convergence for $(q_n)$.
\end{example}
More generally, the condition \eqref{notone} naturally appears in the various well-known results which characterise conditions under which sums of independent random variables converge. In fact, we can extend Example \ref{ex:criterion} to give a computational version of the Kolmogorov's three-series theorem, which would be applicable in situations where we can show that $\sum_{n=0}^\infty \PP(|X_n|\geq a)<\infty$ for some $a>0$, but cannot provide a computable convergence rate for this series.
\begin{theorem}
\label{thm:threeseries}
Let $(X_n)$ be a sequence of independent random variables and define $S_n:=\sum_{i=0}^n X_i$. Suppose that the following conditions are satisfied for some $a>0$, $n_0,q\in\NN$ and $\varphi,\xi:\NN\to \NN$, where $Y_n:=X_nI_{|X_n|\leq a}$:
\begin{enumerate}[(i)]

	\item $\sum_{i=n_0}^\infty \PP(|X_i|\geq a)\leq 1-2^{-q}$;
	
	\item $\left\vert\sum_{i=n}^m \EE[Y_i]\right\vert<2^{-p}$ for all $p\in\NN$ and $m,n\geq \varphi(p)$;
	
	\item $\sum_{n=\xi(p)}^\infty \var(Y_n)<2^{-p}$ for all $p\in\NN$.

\end{enumerate}
Then $(S_n)$ converges almost surely, and for any $\varepsilon>0$, $p\in\NN$ and $g:\NN\to\NN$ with $g(k)\geq k$ for all $k\in\NN$, there exists some $n\leq \tilde g^{(\lfloor (q+1)\log(2)/\varepsilon\rfloor+1)}(\phi(p))$ such that
\[
\PP\left(\exists i,j\in [n;g(n)](|S_i-S_j|\geq 2^{-p}\right)<\varepsilon,
\]
where $\phi(p):=\max\{n_0,\varphi(p+1),\xi(2p+q+5)\}$.
\end{theorem}
\begin{proof}
We show that \eqref{notone} holds for any $p\in\NN$ and for $\lambda:=2^{-(q+1)}$ and $\phi$ as defined in the statement of the result, from which the conclusion follows immediately. For $p,N\in\NN$ and $(U_n)$ some sequence of random variables, let us introduce the notation
\[
P^{p,N}_{(U_n)}:=\exists i,j\geq N(|U_i-U_j|\geq 2^{-p}).
\]
By a straightforward pointwise argument using the triangle inequality we can show that for any $p,N\in\NN$ and sequences $(U_n),(V_n)$:
\begin{equation}
P^{p,N}_{(U_n+V_n)}\subseteq P^{p+1,N}_{(U_n)}\cup P^{p+1,N}_{(V_n)}.
\label{sum}
\end{equation}
Let us also define $Q^N:=\forall i\geq N(|X_i|\leq a)$, and note that by (i) we have $\PP(\overline{Q^N})\leq 1-2^{-q}$ whenever $N\geq n_0$. Finally, let $T_n:=\sum_{i=0}^n Y_i$. Now, observing that on $Q^N$ we have $\forall i\geq N(X_i=Y_i)$ and therefore in particular $\forall i,j\geq N((S_i-S_j)=(T_i-T_j))$, it follows that
\begin{align*}
\PP\left(P^{p,N}_{(S_n)}\right)&= \PP\left(P^{p,N}_{(S_n)}\cap Q^N\right)+\PP\left(P^{p,N}_{(S_n)}\cap\overline{Q^N}\right) \\
&\leq \PP\left(P^{p,N}_{(T_n)}\right)+\PP\left(\overline{Q^N}\right)\leq \PP\left(P^{p,N}_{(T_n)}\right)+1-2^{-q}
\end{align*}
whenever $N\geq n_0$. Next, using \eqref{sum} we note that for any $p,N\in\NN$
\[
\PP\left(P^{p,N}_{(T_n)}\right)\leq \PP\left(P^{p+1,N}_{(T_n-\EE[T_n])}\right)+\PP\left(P^{p+1,N}_{(\EE[T_n])}\right).
\]
But for any $i,j\geq N\geq \varphi(p+1)$, assuming w.l.o.g.\ $j\geq i$, we have
\[
|\EE[T_i]-\EE[T_j]|\leq \left\vert \sum_{k=i+1}^j \EE[Y_i]\right\vert<2^{-(p+1)}
\]
and thus $P^{p+1,N}_{(\EE[T_n])}=\emptyset$, for any $p\in\NN$. Now, again by a pointwise argument via the triangle inequality we have
\begin{align*}
P^{p+1,N}_{(T_n-\EE[T_n])}&\subseteq \exists i\geq N\left(|(T_i-\EE[T_i])-(T_N-\EE[T_N])|\geq 2^{-(p+2)}\right)\\
&=\exists i\geq N\left(\left\vert\sum_{k=N+1}^i(Y_k-\EE[Y_k])\right\vert\geq 2^{-(p+2)}\right).
\end{align*}
By the Kolmogorov maximal inequality, setting $Z_i:=\sum_{k=N+1}^i(Y_k-\EE[Y_k])$ and noting that the $(Y_k)$ are independent, for any $l\in\NN$ we have
\begin{align*}
\PP\left(\exists i\in [N;N+l]\left(\left\vert Z_i\right\vert\geq 2^{-(p+2)}\right)\right)&=\PP\left(\max_{N\leq i\leq N+l}\left\vert Z_i\right\vert\geq 2^{-(p+2)}\right)\\
&\leq 2^{2(p+2)}\sum_{k=N+1}^{N+l}\var(Y_k)\\
&\leq 2^{2(p+2)}\sum_{k=N+1}^{\infty}\var(Y_k)
\end{align*}
and since $l$ was arbitrary
\[
\PP\left(\exists i\geq N\left(\left\vert Z_i\right\vert\geq 2^{-(p+2)}\right)\right)\leq 2^{2(p+2)}\sum_{k=N+1}^{\infty}\var(Y_k).
\]
Therefore whenever $N\geq\xi(2p+q+5)$ it follows that
\[
\PP\left(P^{p+1,N}_{(T_n-\EE[T_n])}\right)\leq 2^{2(p+2)}\sum_{k=\xi(2p+q+5)}^\infty\var[Y_k]<2^{2(p+2)}2^{-(2p+q+5)}=2^{-(q+1)}.
\]
Putting everything together, we have shown that for any $p\in\NN$, if $N\geq \max\{n_0,\varphi(p+1),\xi(2p+q+5)\}$ then
\begin{align*}
\PP\left(P^{p,N}_{(S_n)}\right)&\leq \PP\left(P^{p,N}_{(T_n)}\right)+1-2^{-q}\\
&\leq \PP\left(P^{p+1,N}_{(T_n-\EE[T_n])}\right)+1-2^{-q}\\
&\leq 2^{-(q+1)}+1-2^{-q}=1-2^{-(q+1)}
\end{align*}
from which the claim and hence the result follows.
\end{proof}

\subsection{Bond percolation}
\label{sec:percolation}

Our final example touches on percolation theory, an area in which zero-one laws traditionally play an important role. Here our purpose is simply to show how our zero-one law applies in this setting, and we leave more advanced case studies to future work. For formal details of everything that follows, the reader is directed to e.g.\ \cite{grimmett:99:book} or \cite[Section 2.4]{Kle2020}.

Let $\ZZ^d$ denote the $d$-dimensional integer lattice, and $E^d$ the set of edges $(x,y)$ between neighbours $x,y$ i.e. points $x,y\in \ZZ^d$ with $\norm{x-y}_2=1$. For any parameter $p\in [0,1]$, we consider each edge to be \emph{open} with probability $p$, and \emph{closed} otherwise, independently of all other edges. Writing $E^d_p:=\{e\in E^d\mid \text{$e$ is open}\}$, the structure $(\ZZ^d,E^p)$ is then a random subgraph of $(\ZZ^d,E^d)$ known as a \emph{percolation model}. The intended physical interpretation of $(\ZZ^d,E^d_p)$ is a porous medium where the open edges represent channels along which water can flow. 

To reason formally about $(\ZZ^d,E^d_p)$, ones can work in the probability space $(\Omega,\mathcal{F},\PP_p)$ where $\Omega:=\prod_{e\in E^d}\{0,1\}$, $\mathcal{F}$ is the $\sigma$-algebra generated by finite dimensional cylinders, and $\PP_p$ is the product measure
\[
\PP_p:=\prod_{e\in E^d}\mu^p_e
\]
for $\mu^p_e$ a Bernoulli measure on $\{0,1\}$ with parameter $p$. Concretely, for $\omega\in \Omega$, $\omega(e)=1$ then corresponds to the edge $e$ being open, and
\[
\mu^p_e(\omega(e)=1)=1-\mu^p_e(\omega(e)=0)=p
\]
for all $e\in E^d$. The main questions around percolation models focus the existence of an infinite connected cluster of open edges, and here an application of the Kolmogorov zero-one law tells us that for any parameter, an infinite connected cluster exists with probability zero or one. Since the probability of an infinite connected cluster is monotone in $p$, the more specific question arises, for which \emph{critical} value of $p$ do we experience a phase transition where this probability jumps from zero to one? In the case of the square lattice $\ZZ^2$, a proof that this critical probability is $1/2$ was given in a celebrated paper of Kesten \cite{kesten:80:half}, and more generally, the study of critical values relating to zero-one laws is central to percolation theory.

To see how bond percolation fits within our framework, let us define the $n$-square $S_n\subset \ZZ^d$ by
\[
S_n:=\{(x^1,\ldots,x^d)\in \ZZ^d\mid \text{$|x^j|=n$ for some $j$ and $|x^i|\leq n$ for all $i$}\}.
\]
We say that $x,y\in \ZZ^d$ are path connected, or there is a path between $x$ and $y$, if there exists a sequence $x_0,\ldots,x_k$ of neighbouring points such that $x_0=x$, $x_k=y$ and $(x_i,x_{i+1})$ is open for all $i=0,\ldots,k-1$. Similarly, we say that $S_n$ and $S_k$ are connected if there is a path between $x$ and $y$ for some $x\in S_n$ and some $y\in S_k$. For any parameter $p$, we now set
\[
B^p_{n,k}:=\text{$S_n$ and $S_k$ are not connected}.
\]
\begin{lemma}
\label{lem:perc}
For any $p\in [0,1]$ we have
\[
\overline{B^p_\infty}=\text{there exists an infinite connected cluster},
\]
where $B^p_\infty$ is defined as in Theorem \ref{thm:zeroone}.
\end{lemma}
\begin{proof}
In one direction, suppose there exists an infinite connected cluster $C\subseteq \ZZ^d$. Let $n\in\NN$ be such that $S_n\cap C\neq \emptyset$ and pick $x\in S_n\cap C$. Then for any $k\geq n$, since $x\in C$ there is at least one $y\in S_k$ such that $y\in C$, and therefore there is a path from $x$ to $y$. Therefore $\forall k\geq n\, \overline{B^p_{n,k}}$ and therefore $\overline{B^p_\infty}$. Conversely, suppose that $\overline{B^p_\infty}$ holds, and thus for some $n\in\NN$ we have $\forall k\geq n\, \overline{B^p_{n,k}}$. In other words, for all $k\geq n$ there exists $x_k\in S_n$ and $y_k\in S_k$ such that $x_k$ and $y_k$ are connected. Since $S_n$ is finite there exists an $x\in S_n$ along with a subsequence $n\leq k_0<k_1<\ldots$ such that $y_{k_l}\in S_{k_l}$ and $x$ and $y_{k_l}$ are connected for all $l\in\NN$. Then the set of points $C:=\{x\}\cup\{y_{k_l}\mid l\in\NN\}$ are all connected, and since $C$ is infinite, form part of an infinite connected cluster.
\end{proof}
Let us now define the functions $\theta$ and $\psi$, widely used in bond percolation, by
\begin{align*}
\theta(p)&:=\PP_p\left(\text{the origin is part of an infinite connected cluster}\right),\\
\psi(p)&:=\PP_p\left(\text{there exists an infinite connected cluster}\right).
\end{align*}
It follows immediately from Theorem \ref{thm:zeroone} that for any $p\in\NN$, $\psi(p)\in \{0,1\}$, and moreover if $\theta(p)>0$ then $\psi(p)=1$. We can now provide quantitative analogues of both of these statements. We start with a full finitization of the former:
\begin{theorem}
\label{percolation1}
Fix $\varepsilon,\lambda\in (0,1)$, $r\in\NN$ and $g:\NN\to\NN$. Suppose that each edge between neighbouring points $x,y\in \bigcup_{n=r}^s S_n$ for $s:=b^{r,g}_{\lfloor \log(1/\lambda)/\varepsilon\rfloor+1}$ (for $(b^{r,g}_n)$ defined as in Theorem \ref{thm:main}) is open with probability $p$, independent of the other edges. Then either there exists $n\leq \tilde g^{(\lfloor \log(1/\lambda)/\varepsilon\rfloor+1)}(r)$ such that
\[
\PP_p\left(\text{$S_n$ and $S_{g(n)}$ are connected~}\right)>1-\varepsilon
\]
or 
\[
\PP_p\left(\text{$S_r$ and $S_s$ are connected~}\right)<\lambda.
\]
\end{theorem}
\begin{proof}
Directly from Corollary \ref{cor:independent}. To see that $(B^p_{n,k})$ satisfies the required conditions, note that for $n\leq m<l\leq k$, we have $\overline{B^p_{n,k}}\subseteq \overline{B^p_{n,m}}\cap\overline{B^p_{l,k}}$ since whenever there exists a path from $S_n$ to $S_k$, that must include a subpath from $S_n$ to $S_m$, together with a subpath from $S_l$ to $S_k$. The independence condition is clear.
\end{proof}

\begin{corollary}
\label{percolation2}
Consider the percolation model $(\ZZ^d,E^p)$ for $p\in [0,1]$. Suppose that $\theta(p)\geq \lambda_p$ for some $\lambda_p>0$. Then for any $\varepsilon\in (0,1)$ and $g:\NN\to\NN$, there exists $n\leq \tilde g^{(\lfloor \log(1/\lambda_p)/\varepsilon\rfloor+1)}(0)$ such that
\[
\PP_p\left(\text{$S_n$ and $S_{g(n)}$ are connected~}\right)>1-\varepsilon.
\]
\end{corollary}
\begin{proof}
By similar reasoning to Lemma \ref{lem:perc}, one can show that the origin being part of an infinite connected cluster is equivalent to $\forall k\, \widebar{B^p_{0,k}}$, and therefore for any $s\in\NN$:
\[
\PP_p\left(\text{$S_0$ and $S_s$ are connected~}\right)=\PP_p\left(\widebar{B^p_{0,s}}\right)\geq \PP_p\left(\forall k\, \widebar{B^p_{0,k}}\right)=\theta(p)\geq\lambda_p,
\]
and so the result follows from Theorem \ref{percolation1}, setting $\lambda:=\lambda_p$ and $r:=0$.
\end{proof}
It should be noted that in the case of Corollary \ref{percolation2}, better bounds on $n$ can often be read off from the literature. For example, Theorem 2 of \cite{kesten:80:half} provides, in the case $p>1/2$, upper bounds on the probability that the infinite connected cluster contains no vertices within distance $n$ of the origin, which would then result in a bound on $n$ that is even independent of $g$. However, unlike the results of \cite{kesten:80:half}, which require significant machinery to establish, Corollary \ref{percolation2} is completely elementary, and applies uniformly in arbitrary dimensions. 

We consider Theorem \ref{percolation1} to be of interest in its own right, as a full finitisation of the statement that $\psi(p)\in\{0,1\}$, equivalent in strength to the latter but formulated entirely in terms of finite regions of the lattice $\ZZ^d$. Indeed, Theorem \ref{percolation1} bears a passing resemblance to Russo's approximate zero-one law \cite{russo:82:zeroone} (see also \cite{talagrand:94:zeroone}), in that we replace an infinite model with a finite one, and establish that an event over this finite model holds with probability close to one or zero (however our result differs in that we work in a more abstract setting -- with Theorem \ref{percolation1} an instance of a general finitary zero-one law -- whereas Russo's is specific to percolation and also provides information about phase transitions).

\section{Outlook}

We have initiated the study of zero-one laws from a quantitative and finitistic perspective, combining abstract results with a series of simple applications across various representative areas where zero-one laws traditionally appear. We propose that further work in this direction would be of considerable interest, from both a purely logical perspective and also with an eye towards serious applications in probability theory.

Given their inherent set-theoretic complexity and nonconstructive nature, the analysis of stronger zero-one laws through e.g.\ the Dialectica interpretation would be of conceptual value, as a testing ground for the applicability of logical methods in probability, likely requiring new theoretical developments such as a tame way of treating $\sigma$-algebras. The zero-one law considered here can be treated relatively easily as it imposes a fixed logical structure on tail events, something that is certainly not guaranteed in the general Kolmogorov zero-one law, where tail events can be of arbitrary complexity in terms of infinite unions and intersections, the latter being one of the main challenges in applying the Dialectica in probability \cite{neri-pischke:pp:formal}. Were progress to be made on this front, a subsequent analysis of the more general L\'evy's zero–one law could potentially make use of recent insights gained in the quantitative study of martingales \cite{neri-powell:25:martingale}.

In terms of applications outside of logic, the process of finitising zero-one laws, including the one treated here, gives us a method for eliminating them when they appear as part of a larger proof, in line with the modular nature of the (classical) Dialectica interpretation. It is an open question whether either the finitary zero-one law given here, or a future extension thereof, could be used to obtain new quantitative results that are of direct relevance in probability theory, and we speculate that percolation theory in particular may represent a promising area to explore in this regard.\\

\noindent\textbf{Acknowledgements.} The authors thank Nicholas Pischke and Paulo Oliva for interesting comments and discussions, and suggesting many improvements to an earlier draft of this work, and Paulo Oliva for clarifying aspects of \cite{arthan-oliva:21:borel-cantelli} and communicating a shortened version of one of the proofs, included here as the proof of Lemma \ref{lem:er}. The first author was partially supported by an EPSRC grant EP/W035847/1, and the second by an EPSRC DTG training grant  EP/W524712/1.

\bibliographystyle{acm}
\bibliography{../tpbiblio,awlib}

\begin{thebibliography}{10}

\bibitem{arthan-oliva:21:borel-cantelli}
{\sc Arthan, R., and Oliva, P.}
\newblock On the {B}orel-{C}antelli {L}emma, the {E}rd\'os-{R}\'enyi {T}heorem,
  and the {K}ochen-{S}tone {T}heorem.
\newblock {\em Journal of Logic and Analysis 13}, 6 (2021), 1--23.

\bibitem{avigad-dean-rute:12:dominated}
{\sc Avigad, J., Dean, E.~T., and Rute, J.}
\newblock A metastable dominated convergence theorem.
\newblock {\em Journal of Logic \& Analysis 4}, 3 (2012), 1--19.

\bibitem{avigad-gerhardy-towsner:10:local}
{\sc Avigad, J., Gerhardy, P., and Towsner, H.}
\newblock Local stability of ergodic averages.
\newblock {\em Transactions of the American Mathematical Society 362}, 1
  (2010), 261--288.

\bibitem{erdos-renyi:59:cantor}
{\sc Erd\H{o}s, P., and R\'enyi, A.}
\newblock On {C}antor's series with convergent $\sum 1/q_n$.
\newblock {\em Annales Universitatis Scientiarium Budapestinensis de Rolando
  E\"{o}tv\"{o}s Nominatae Sectio Mathematica 2\/} (1959), 93--109.

\bibitem{Fagin:76:probabilitiesonfinitemodels}
{\sc Fagin, R.}
\newblock Probabilities on finite models.
\newblock {\em Journal of Symbolic Logic 41}, 1 (1976), 50--58.

\bibitem{gaspar-kohlenbach:10:pigeonhole}
{\sc Gaspar, J., and Kohlenbach, U.}
\newblock On {T}ao's ``finitary'' infinite pigeonhole principle.
\newblock {\em Journal of Symbolic Logic 75\/} (2010), 355--371.

\bibitem{GlebskiiKoganLiogon'kii:72:realizabilitypredicatecalculus}
{\sc Glebskii, Y.~V., Kogan, D.~I., Liogon’kii, M.~I., and Talanov, V.~A.}
\newblock Range and degree of realizability of formulas in the restricted
  predicate calculus.
\newblock {\em Cybernetics 5}, 2 (1972), 142--154.

\bibitem{goedel:58:dialectica}
{\sc G\"odel, K.}
\newblock {\"Uber eine bisher noch nicht ben\"utzte Erweiterung des finiten
  Standpunktes}.
\newblock {\em dialectica 12}, 3--4 (1958), 280--287.

\bibitem{grimmett:99:book}
{\sc Grimmett, G.}
\newblock {\em Percolation}.
\newblock Springer, 1999.

\bibitem{GraedelHelalNaafWilke:22:zeroonesemiring}
{\sc Grädel, E., Helal, H., Naaf, M., and Wilke, R.}
\newblock Zero-one laws and almost sure valuations of first-order logic in
  semiring semantics.
\newblock In {\em Proceedings of Logic in Computer Science (LICS'22)\/} (2022),
  arXiv.

\bibitem{kesten:80:half}
{\sc Kesten, H.}
\newblock The critical probability of bond percolation on the square lattice
  equals 1/2.
\newblock {\em Communications in Mathematical Physics 74\/} (1980), 41--89.

\bibitem{Kle2020}
{\sc Klenke, A.}
\newblock {\em Probability Theory}, third~ed.
\newblock Springer, 2020.

\bibitem{kohlenbach:08:book}
{\sc Kohlenbach, U.}
\newblock {\em {Applied Proof Theory: Proof Interpretations and their Use in
  Mathematics}}.
\newblock Springer Monographs in Mathematics. Springer, 2008.

\bibitem{kohlenbach:19:nonlinear:icm}
{\sc Kohlenbach, U.}
\newblock Proof-theoretic methods in nonlinear analysis.
\newblock In {\em Proceedings of the International Congress of Mathematicians
  2018}, vol.~2. World Scientific, 2019, pp.~61--82.

\bibitem{kreuzer:15:measure}
{\sc Kreuzer, A.}
\newblock {Measure theory and higher order arithmetic}.
\newblock {\em Proceedings of the American Mathematical Society 143}, 12
  (2015), 5411--5425.

\bibitem{Ner2024}
{\sc Neri, M.}
\newblock {Quantitative strong laws of large numbers}.
\newblock {\em Electronic Journal of Probability 30}, 20 (2024).
\newblock 22pp.

\bibitem{neri:25:kronecker}
{\sc Neri, M.}
\newblock {A finitary Kronecker's lemma and large deviations in the strong law
  of large numbers on Banach spaces}.
\newblock {\em Annals of Pure and Applied Logic 176}, 6 (2025), 103569.

\bibitem{neri-pischke:pp:formal}
{\sc Neri, M., and Pischke, N.}
\newblock Proof mining and probability theory.
\newblock {\em ArXiv e-prints\/} (2024).
\newblock arXiv, math.LO, 2403.00659.

\bibitem{neri-pischke-powell:25:learnability}
{\sc Neri, M., Pischke, N., and Powell, T.}
\newblock Generalized learnability of stochastic principles.
\newblock In {\em Proceedings of Computability in Europe (CiE'25)\/} (2025),
  vol.~15764 of {\em LNCS}, pp.~333--348.

\bibitem{NPP2025a}
{\sc Neri, M., Pischke, P., and Powell, T.}
\newblock {On the asymptotic behaviour of stochastic processes, with
  applications to supermartingale convergence, Dvoretzky's approximation
  theorem, and stochastic quasi-Fejér monotonicity}.
\newblock {\em ArXiv e-prints\/} (2024).
\newblock arXiv, math.OC, 2504.12922.

\bibitem{neri-powell:pp:rs}
{\sc Neri, M., and Powell, T.}
\newblock A quantitative {R}obbins-{S}iegmund theorem.
\newblock {\em ArXiv e-prints\/} (2024).
\newblock arXiv, math.OC, 2410.15986.

\bibitem{neri-powell:25:martingale}
{\sc Neri, M., and Powell, T.}
\newblock On quantitative convergence for stochastic processes: Crossings,
  fluctuations and martingales.
\newblock {\em Transactions of the American Mathematical Society, Series B
  12\/} (2025), 974--1019.

\bibitem{PP2024}
{\sc Pischke, P., and Powell, T.}
\newblock {Asymptotic regularity of a generalised stochastic Halpern scheme
  with applications}.
\newblock {\em ArXiv e-prints\/} (2024).
\newblock arXiv, math.OC, 2411.04845.

\bibitem{russo:82:zeroone}
{\sc Russo, L.}
\newblock An approximate zero-one law.
\newblock {\em Zeitschrift f\"ur Wahrscheinlichkeitstheorie und verwandte
  Gebiete 61\/} (1982), 129--139.

\bibitem{shiryaev:book}
{\sc Shiryaev, A.~N.}
\newblock {\em Probability}.
\newblock Graduate Texts in Mathematics. Springer New York, NY, 1996.

\bibitem{shoenfield:67:book}
{\sc Shoenfield, J.}
\newblock {\em {Mathematical Logic}}.
\newblock Addison-Wesley Publishing Co., 1967.

\bibitem{specker:49:sequence}
{\sc Specker, E.}
\newblock {N}icht konstruktiv beweisbare {S}{\"a}tze der {A}nalysis.
\newblock {\em Journal of Symbolic Logic 14\/} (1949), 145--158.

\bibitem{talagrand:94:zeroone}
{\sc Talagrand, M.}
\newblock On {R}usso's approximate zero-one law.
\newblock {\em Annals of Probability 22}, 3 (1994).

\bibitem{tao:07:softanalysis}
{\sc Tao, T.}
\newblock Soft analysis, hard analysis, and the finite convergence principle.
\newblock Essay posted 23 May 2007, 2007.
\newblock Appeared in: ‘T. Tao, Structure and Randomness: Pages from Year One
  of a Mathematical Blog. AMS, 298pp., 2008’.

\end{thebibliography}

\end{document}